\documentclass[11pt,a4paper]{amsart}
\usepackage {mypackage}
\date{\today}

\author{Nicolas Tholozan}
\address{University of Luxembourg, Campus Kirchberg \\
6, rue Richard Coudenhove-Kalergi\\
L-1359 Luxembourg}
\email{nicolas.tholozan@uni.lu}

\title[The Volume of complete AdS $3$-manifolds]{The Volume of complete anti-de Sitter $3$-manifolds}

\begin{document}

\maketitle

\begin{abstract}
Up to a finite cover, closed anti-de Sitter $3$-manifolds are quotients of $\SO_0(2,1)$ by a discrete subgroup of $\SO_0(2,1) \times \SO_0(2,1)$ of the form
\[j\times \rho(\Gamma)~,\]
where $\Gamma$ is the fundamental group of a closed oriented surface, $j$ a Fuchsian representation and $\rho$ another representation which is ``strictly dominated'' by $j$.

Here we prove that the volume of such a quotient is proportional to the sum of the Euler classes of $j$ and $\rho$. As a consequence, we obtain that this volume is constant under deformation of the anti-de Sitter structure. Our results extend to (not necessarily compact) quotients of $\SO_0(n,1)$ by a discrete subgroup of $\SO_0(n,1) \times \SO_0(n,1)$.
\end{abstract}

\section*{Introduction}

\subsection*{Volume of non-Riemannian Clifford--Klein forms}

It is a well-known fact that the covolume of a lattice $\Gamma$ in a simple Lie group $G$ is strongly rigid in the following sense: for any two embeddings $j$ and $j': \Gamma \to G$ as a lattice, the quotients $j(\Gamma) \backslash G$ and $j'(\Gamma) \backslash G$ have the same volume (with respect to some fixed Haar measure).

This fact is actually a by-product of the famous rigidity theorems of Mostow \cite{Mostow68} and Margulis \cite{Margulis91}. The only situation that is not covered by those theorems is when $G$ is (a covering of) $\PSL(2,\R)$, in which case the result follows from the Gauss--Bonnet formula (the covolume of $j(\Gamma)$ being proportional to the volume of $j(\Gamma) \backslash \H^2$).

Comparatively, very little is known about the volume rigidity of \emph{non-Riemannian Clifford--Klein forms}, i.e. manifolds of the form $\Gamma \backslash G/H$, where $H$ is a closed non-compact subgroup of $G$ and $\Gamma$ a discrete subgroup of $G$. In this case, there is no analog of Mostow's and Margulis's rigidity theorems. The Gauss--Bonnet formula can only be generalized for manifolds of even dimension. In some interesting cases, it actually gives an obstruction to the existence of compact Clifford--Klein forms.

To illustrate the difficulties that arrise, let us focus on the case of compact anti-de Sitter manifolds (i.e. Lorentz manifolds of constant negative sectional curvature). Those manifolds have (up to finite cover) the form $\Gamma \backslash \AdS_n$, where $\AdS_n$ is the homogeneous space $\SO_0(n-1,2)/\SO_0(n-1,1)$ and $\Gamma$ is a discrete subgroup of $\SO_0(n-1,2)$ acting freely, properly discontinously and cocompactly on $\AdS_n$.

When $n$ is even, the Chern--Gauss--Bonnet formula implies that the volume of $\Gamma \backslash \AdS_n$ should be proportional to its Euler characteristic. More than a rigidity result, this actually implies the non-existence of such a group $\Gamma$, simply because the Euler characteristic of a closed Lorentz manifold is always~$0$.

When $n= 2k+1$ is odd and greater than $4$, the only known examples of such groups $\Gamma$ are lattices in $\U(k,1) \subset \SO_0(2k,2)$ (these examples were given by Kulkarni in \cite{Kulkarni81}). Such a group $\Gamma$ cannot be continuously deformed outside $\U(k,1)$ (see \cite{KasselThese}, paragraph 6.1.1) and the volume of $\Gamma \backslash \AdS_{2k+1}$ is thus locally rigid. It is conjectured that these are the only examples, which would imply a srong volume rigidity.

In contrast, compact quotients of $\AdS_3$ are known to have a rich deformation space (see \cite{Goldman85}, \cite{Salein00}, \cite{KasselThese} and \cite{Tholozan3}). In a recent survey by anti-de Sitter geometers, it is asked whether the volume of closed anti-de Sitter $3$-manifolds is constant under these deformations (\cite{QuestionsAdS}, question 2.3). The primary purpose of this paper is to answer this question.

The specificity of anti-de Sitter manifolds in dimension $3$ is that $\AdS_3$ can be identified with $\SO_0(2,1)$ with its Killing metric and $\Isom_0(\AdS_3)$ with $\SO_0(2,1) \times \SO_0(2,1)$ acting of $\SO_0(2,1)$ by left and right multiplication. By work of Kulkarni--Raymond \cite{KulkarniRaymond85} and Klingler \cite{Klingler96}, compact anti-de Sitter $3$-manifolds have (up to a finite cover) the form
\[j\times \rho(\Gamma) \backslash \SO_0(2,1)~,\]
where $\Gamma$ is the fundamental group of a closed oriented surface and $j$ and $\rho$ are two representations of $\Gamma$ into $\SO_0(2,1)$, with $j$ discrete and faithful.

We will prove that the volume of such a manifold is proportional to the sum of the \emph{Euler classes} of $j$ and $\rho$, implying local rigidity of the volume. More generally, we will describe the volume of quotients of $\SO_0(n,1)$ by a finitely generated subgroup of $\SO_0(n,1) \times \SO_0(n,1)$.

\subsection*{Statement of the results}

The action of $\SO_0(n,1) \times \SO_0(n,1)$ on $\SO_0(n,1)$ by left and right multiplication preserves a volume form on $\SO_0(n,1)$. This allows to define the volume of quotients of $\SO_0(n,1)$ by discrete subgroups of $\SO_0(n,1) \times \SO_0(n,1)$ acting properly discontinuously.

Kobayashi \cite{Kobayashi93} and Kassel \cite{Kassel08}, generalizing the work of Kulkarni--Raymond \cite{KulkarniRaymond85}, proved that finitely generated subgroups of $\SO_0(n,1) \times \SO_0(n,1)$ acting properly discontinously on $\SO_0(n,1)$ have (up to finite index and up to switching the factors) the form
\[j\times \rho(\Gamma)~,\]
where $\Gamma$ is a finitely generated group and $j$ and $\rho$ two representations of $\Gamma$ into $\SO_0(n,1)$, with $j$ discrete and faithful (see theorem \ref{t:KobayashiKassel}). Here, we will compute the volume of these quotients.

Recall that $\SO_0(n,1)$ is the group of orientation preserving isometries of the hyperbolic space $\H^n$. If $j$ is discrete and faithful, then $j(\Gamma)$ acts properly discontinuously on $\H^n$. If, moreover, this action is cocompact, then one can define the volume of a representation $\rho: \Gamma \to \SO_0(n,1)$ by
\[\Vol(\rho) = \int_{j(\Gamma) \backslash \H^n} f^*\vol^{\H^n}~,\]
where $f$ is any piecewise smooth $(j,\rho)$-equivariant map from $\H^n$ to $\H^n$. Here $\vol^{\H^n}$ denotes the volume form associated to the hyperbolic metric on $\H^n$ (see section \ref{ss:Conventions} for precisions on its normalization). In particular,
\[\Vol(j) = \Vol\left(j(\Gamma) \backslash \H^n\right)~.\]
This definition can be extended (with some technicality) to representations of a non-cocompact hyperbolic lattice (see definition \ref{d:VolumeRepresentation} and the discussion in section \ref{ss:VolumeRepresentation}).

Finally, let us denote by $\V_n$ the volume of $\SO(n)$ with respect to its bi-invariant volume form (see section \ref{ss:Conventions} for the normalization of this volume form). With these definitions and notations, we can now state our main result:

\begin{MonThm} \label{t:VolumeGKQuotients}
Let $\Gamma$ be a lattice in $SO_0(n,1)$, $j: \Gamma \to \SO_0(n,1)$ the inclusion and $\rho$ another representation of $\Gamma$ into $\SO_0(n,1)$, such that $j\times \rho(\Gamma)$ acts properly discontinuously on $\SO_0(n,1)$. 

Then \[\Vol\left( j\times \rho(\Gamma) \backslash \SO_0(n,1) \right) = \V_n \left( \Vol(j) + (-1)^n \Vol(\rho) \right)~.\]
\end{MonThm}

\begin{rmk}
We will give later a slightly stronger version of this theorem (Theorem \ref{t:MainThmBis}).
\end{rmk}

The proof will rely on a theorem of Gu\'eritaud--Kassel \cite{GueritaudKassel} which gives a necessary and sufficient condition for $j\times \rho(\Gamma)$ to act properly discontinuously on $\SO_0(n,1)$ (Theorem \ref{t:GueritaudKassel}). This theorem allows a deep understanding of the geometry of the quotient (see section \ref{ss:DescriptionGKQuotients}). Gu\'eritaud--Kassel's condition is open, meaning that if $j_0\times \rho_0(\Gamma)$ acts properly discontinuously on $\SO_0(n,1)$, then $j\times \rho(\Gamma)$ still acts properly discontinuously for $(j,\rho)$ in some neighbourhood of $(j_0,\rho_0)$. We will say that the volume of $j_0 \times \rho_0(\Gamma) \backslash \SO_0(n,1)$ is \emph{locally rigid} if the volume of $j \times \rho(\Gamma) \backslash \SO_0(n,1)$ is constant for $(j,\rho)$ in a neighbourhood of $(j_0,\rho_0)$.

Using rigidity results of Besson--Courtois--Gallot \cite{BCG07} and Kim--Kim \cite{KimKim13} for the volume of representations of a lattice in $\SO_0(n,1)$, we obtain the following corollary:

\begin{MonCoro} \label{c:VolumeRigidity}
Let $\Gamma$ be a lattice in $\SO_0(n,1)$, $j:\Gamma \to \SO_0(n,1)$ the inclusion and $\rho: \Gamma \to \SO_0(n,1)$ a representation of $\Gamma$ into $\SO_0(n,1)$ such that $j\times \rho(\Gamma)$ acts properly discontinuously on $\SO_0(n,1)$.
Then the volume of $j\times \rho(\Gamma) \backslash \SO_0(n,1)$ is rigid, unless $n=2$ and $\Gamma$ is not cocompact.
\end{MonCoro}

The case when $n=2$ corresponds to closed anti-de Sitter $3$-manifolds and thus answers Question 2.3 in \cite{QuestionsAdS}. Interestingly, in that case, the volume is not rigid when $\Gamma$ is not cocompact. We will prove the following:

\begin{MonThm} \label{t:NonRigidity}
Let $\Gamma$ be a non-cocompact lattice in $\SO_0(2,1)$ and $j:\Gamma \to \SO_0(2,1)$ the inclusion. Then there exists a continuous family $\left(\rho_t\right)_{-1<t<1}$ of representations of $\Gamma$ into $\SO_0(2,1)$ such that $j\times \rho_t(\Gamma)$ acts properly discontinuously on $\SO_0(2,1)$ and such that 
\[\Vol\left(j\times \rho_t(\Gamma) \backslash \SO_0(2,1) \right)\]
takes all the values in the open interval $(0, 4\pi \Vol(j))$ when $t$ varies.
\end{MonThm}

\subsection*{Other approaches to theorem \ref{t:VolumeGKQuotients}}

Theorem \ref{t:VolumeGKQuotients} was originally proven in the author's thesis \cite{TholozanThese} for $n=2$ and for compact quotients. Since then, it attracted some interest, and several people found other approaches to prove it in this setting.

Jean-Marc Schlenker mentioned to us that the volume rigidity could probably be obtained using the anti-de Sitter version of the Schl\"afli formula (see proposition \ref{p:SchlafliFormula}).

Labourie found an interpretation of the volume of anti-de Sitter $3$-manifolds as a Chern--Simons invariant \cite{LabourieMSRI}. He noticed that the volume could be understood as a secondary characteristic class of the tangent bundle provided with two natural flat connections coming from left and right parallelism on $\SO_0(2,1)$. The general Chern--Simons theory then directly gives the rigidity of the volume. This approach may generalize to quotients of $\SO_0(n,1)$, with $n \geq 3$.
 
Finally, Alessandrini and Li recently found another proof of this theorem \cite{AlessandriniLi15}. Their computation seems similar to ours but uses an interesting Higgs bundle interpretation of anti-de Sitter $3$-manifolds.

\subsection*{Organization of the paper}

The next section of this paper introduces the background of our work and gives more details about our results. Section \ref{s:AdS} specializes our results in the anti-de Sitter case and proves theorem \ref{t:NonRigidity}. Section \ref{s:Rigidity} gives some precisions on the definition of the volume of a representation and proves corollary \ref{c:VolumeRigidity}. Section \ref{s:ComputationVolume} contains the core of the proof of Theorem \ref{t:GueritaudKassel}. Finally, in section \ref{s:OpenQuestions}, we mention some open questions raised by this work.

\subsection*{Acknowledgements}

Part of this work was completed during my PhD at the university of Nice-Sophia Antipolis. I would like to thank my advisor, Sorin Dumitrescu, for is constant support during my thesis. I am also thankful to François Labourie for his interest in this work and for explaining me his approach with Chern--Simons theory, and to Fanny Kassel and Ravi Kulkarni for some thoughtful remarks after a talk I gave on this topic.

\section{Volume of Gu\'eritaud--Kassel's quotients} \label{s:Background}

\subsection{Hausdorf quotients of $\SO_0(n,1)$}

Let us start by reviewing the previous results describing quotients of $\SO_0(n,1)$ by a discrete subgroup of $\SO_0(n,1) \times \SO_0(n,1)$. From now on, we will name these \emph{Gu\'eritaud--Kassel's quotients}, in reference to Gu\'eritaud--Kassel's theorem (Theorem \ref{t:GueritaudKassel}) that completed their description.

Gu\'eritaud--Kassel's quotients are examples of \emph{Clifford--Klein forms}, that is, quotients of a homogeneous space by a discrete subgroup of transformations acting properly discontinuously. In the case we are interested in, the homogeneous space is the Lie group $\SO_0(n,1)$, the connected component of the identity in the group of linear transformations of $\R^{n+1}$ preserving a quadratic form of signature $(n,1)$. Its transformation group is the group $\SO_0(n,1) \times \SO_0(n,1)$ acting by left and right multiplication:
\[(g,h) \cdot x = g x h^{-1}~.\]

This action preserves a natural pseudo-Riemannian metric on $\SO_0(n,1)$ called the \emph{Killing metric}. To define it, recall that the tangent space of $\SO_0(n,1)$ at the identity is the Lie algebra $\so(n,1)$ of $(n+1) \times (n+1)$ matrices $A$ satisfying
\[\transp{A} \I_{n,1} + \I_{n,1} A = 0~,\]
where
\[\I_{n,1} = 
\left( \begin{matrix}
1 & 0 & \ldots & \ & 0 \\
0 \\
\vdots & \ & \ddots & \ & \vdots \\
\ & \ & \ & 1 & 0\\
0 &\ & \ldots & 0 & -1
\end{matrix} \right)~.\]

\begin{defi}
The Killing metric on $\SO_0(n,1)$ is the left invariant pseudo-Riemmanian metric $\kappa$ defined on $\so(n,1)$ by
\[\kappa(A,B) = \frac{1}{2}\Tr(AB)~.\]
\end{defi}

The adjoint action of $\SO_0(n,1)$ on $\so(n,1)$ preserves the Killing form. Therefore, $\kappa$ is also invariant under right multiplication. Since $\SO_0(n,1)$ is simple, this actually characterizes the Killing metric.

\begin{prop}[Cartan]
The Killing metric is, up to scaling, the unique pseudo-Riemannian metric on $\SO_0(n,1)$ which is invariant by both left and right multiplication. Moreover, $\SO_0(n,1) \times \SO_0(n,1)$ (acting by left and right multiplication) identifies with (an index $4$ subgroup of) the isometry group of $(\SO_0(n,1), \kappa)$.\\
\end{prop}

We now turn to the description of the discrete subgroups of $\SO_0(n,1) \times \SO_0(n,1)$ acting properly discontinuously on $\SO_0(n,1)$. Recall that $\SO_0(n,1)$ can be seen as the group of orientation preserving isometries of the hyperbolic space $\H^n$.

\begin{CiteThm}[Kulkarni--Raymond \cite{KulkarniRaymond85}, Kobayashi \cite{Kobayashi93}, Kassel \cite{Kassel08}] \label{t:KobayashiKassel}
Let $\Gamma$ be a discrete and torsion-free subgroup of $\SO_0(n,1)\times \SO_0(n,1)$. Denote by $j$ (resp. $rho$) the projection of $\Gamma$ on the first (resp. second) factor. If $\Gamma$ acts properly discontinuously on $\SO_0(n,1)$, then, up to switching the factors, $j$ is discrete and faithful.
\end{CiteThm}

\begin{rmk}
This result was initially proven by Kulkarni and Raymond in the case $n=2$ \cite{KulkarniRaymond85}. Kobayashi then proved faithfulness and Kassel proved discreteness in a more general setting.
\end{rmk}

Gu\'eritaud and Kassel then gave a precise criterion for $j\times \rho(\Gamma)$ to act properly discontinuously. In order to state it, let us introduce some terminology first.

\begin{defi}
Let $\Gamma$ be a group, $j: \Gamma \to \SO_0(n,1)$ a discrete and faithful representation and $\rho: \Gamma \to \SO_0(n,1)$ another representation. We say that $j$ \emph{strictly dominates} $\rho$ if there exists a $(j,\rho)$-equivariant map from $\H^n$ to $\H^n$ which is contracting (i.e. $\lambda$-Lipschitz for some $\lambda < 1$).
\end{defi}

\begin{CiteThm}[Gu\'eritaud--Kassel \cite{GueritaudKassel}]\label{t:GueritaudKassel}
Let $\Gamma$ be a geometrically finite subgroup of $\SO_0(n,1)$, $j:\Gamma \to \SO_0(n,1)$ the inclusion and $\rho$ another representation of $\Gamma$ into $\SO_0(n,1)$. Then the group $(j\times \rho)(\Gamma)$ acts (freely and) properly discontinuously on $\SO_0(n,1)$ if and only if, up to switching the factors, $\rho$ is strictly dominated by $j$.
\end{CiteThm}

\begin{rmk}
The fact that the strict domination is a sufficient condition follows from the Benoist--Kobayashi properness criterion. It seems to have been explicitly stated for the first time by Salein \cite{Salein00}.
\end{rmk}

Gu\'eritaud--Kassel's theorem does not hold if we remove the hypothesis that $\Gamma$ is geometrically finite. However, we conjecture that, if $\Gamma$ acts properly discontinuously and cocompactly on $\SO_0(n,1)$ with finite covolume, then $j(\Gamma)$ must be geometrically finite (see section \ref{sss:FinitelyGenerated}). \\

\subsection{Volume}

There is, up to scaling, a unique bi-invariant volume form on $\SO_0(n,1)$. A natural way to normalize it is to decide that we get $1$ when we evaluate it on a direct orthonormal basis of $(\so(n,1),\kappa)$ (this requires first a choice of an orientation). We denote this volume form by $\vol^\kappa$.

Let $\Gamma$ be a finitely generated group and $j$, $\rho$ two representations of $\Gamma$ into $\SO_0(n,1)$ such that $j$ is discrete and faithful and strictly dominates $\rho$. Then the volume form $\vol^\kappa$ on $\SO_0(n,1)$ induces a volume form on its quotient by the action of $\Gamma$.

\begin{defi}
The volume of Gu\'eritaud--Kassel's quotient
\[j\times\rho(\Gamma) \backslash \SO_0(n,1)\]
is defined as
\[\Vol \left( j\times\rho(\Gamma) \backslash \SO_0(n,1) \right) = \int_{j\times\rho(\Gamma) \backslash \SO_0(n,1)} \vol^\kappa~.\]
\end{defi}

Let us reformulate our main theorem in a more precise way. We denote by $\vol^{\H^n}$ the volume form on $\H^n$ associated to the metric of constant curvature $-1$.

\begin{theo} \label{t:MainThmBis}
Let $\Gamma$ be a geometrically finite subgroup of $\SO_0(n,1)$, $j: \Gamma \to \SO_0(n,1)$ the inclusion and $\rho: \Gamma\to \SO_0(n,1)$ a representation which is strictly dominated by $j$. Let $f: \H^n\to \H^n$ be a $(j,\rho)$-equivariant contracting map. Then we have:
\[\Vol\left( (j\times \rho)(\Gamma) \backslash \SO_0(n,1)\right) = \V_n \int_{j(\Gamma) \backslash \H^n} \vol^{\H^n} + (-1)^n f^*\vol^{\H^n}~, \]
where $\V_n$ is the volume of $\SO(n)$ with respect to its Killing metric.
\end{theo}

\begin{coro} \label{c:FiniteVolume}
If $(j\times \rho)(\Gamma) \backslash \SO_0(n,1)$ has finite volume, then $j(\Gamma)$ is a lattice in $\SO_0(n,1)$ and the integral
\[\int_{j(\Gamma) \backslash \H^n} f^*\vol^{\H^n}\]
does not depend on the choice of $f$.
\end{coro}

In that case, one can take $\int_{j(\Gamma) \backslash \H^n} f^*\vol^{\H^n}$ as a definition of the ``volume'' of the representation $\rho$. This coincides with the classical definition when $j(\Gamma)$ is cocompact, and with more subtle definitions introduced in \cite{BBI13} and \cite{Francaviglia04} in general (see section \ref{ss:VolumeRepresentation}). Theorem \ref{t:VolumeGKQuotients} thus follows from theorem \ref{t:MainThmBis}.

\begin{proof}[Proof of corollary \ref{c:FiniteVolume}]
Let $\lambda < 1$ be such that $f$ is $\lambda$-Lipschitz. Then we have
\[| f^*\vol^{\H^n} | \leq \lambda^n \vol^{\H^n}\]
and thus, by Theorem \ref{t:VolumeGKQuotients},
\[\Vol\left( (j\times \rho)(\Gamma) \backslash \SO_0(n,1)\right) \geq (1-\lambda^n) \V_n \int_{j(\Gamma) \backslash \H^n} \vol^{\H^n}~. \]
Therefore, if $(j\times \rho)(\Gamma) \backslash \SO_0(n,1)$ has finite volume, then $j(\Gamma) \backslash \H^n$ has finite volume and $j(\Gamma)$ is thus a lattice. We can then re-write the volume of $(j\times \rho)(\Gamma) \backslash \SO_0(n,1)$ as
\[\int_{j(\Gamma)\backslash \H^n} \vol^{\H^n} + \int_{j(\Gamma) \backslash \H^n} f^*\vol^{\H^n}~.\]
Since this volume does not depend on the contracting map $f$, the second term obviously does not depend on $f$ either.
\end{proof}

\subsection{A word on conventions} \label{ss:Conventions}

All these results depend on the normalization we chose for the Killing metric. The definition we chose is rather specific to $\SO(n,1)$ which has a prefered linear representation (for a general Lie group, one usually uses the adjoint representation). However, this choice is particularly suited to hyperbolic geometry. Indeed, the orthogonal of $\so(n)$ in $\so(n,1)$ naturally identifies to the tangent space to $\H^n$ at some base point $x_0$. With our convention, the restriction of the Killing metric to $\so(n)^\perp$ precisely gives the Riemannian metric of constant curvature $-1$ on $\H^n$.

One also needs to choose a bi-invariant volume form on $\SO(n)$ in order to define $\V_n$. We do this in the same way as we did for $\SO(n,1)$. First define the Killing form on the Lie algebra $\so(n)$ by $(A,B) \mapsto \frac{1}{2}\Tr(AB)$, where $A$ and $B$ are anti-symmetric $n\times n$ matrices (note that this coincides with the restriction to $\so(n)$ of the Killing form on $\so(n,1)$); then normalize the volume form $\vol^{\SO(n)}$ on $\SO(n)$ so that the volume of an oriented orthonormal basis of $\so(n)$ is $1$. With this convention, the volumes $\V_n$ satisfy the nice recurrence relation:
\[\V_{n+1} = V_n \times \Vol(\S^n)~,\]
where $\S^n$ denotes the $n$-sphere. One obtains the following (not very enlightening) expression for $\V_n$:
\[\V_n = \prod_{k=1}^{n-1} \frac{\pi^{n/2}}{\Gamma(n/2)}~,\]
where $\Gamma$ is the Gamma function.

Finally, our results seem to depend on the orientation of $\SO_0(n,1)$ and $\H^n$. The most natural convention to take is to ask for the volume of a manifold to be positive. Note however that, once we fixed an orientation on $\H^n$ so that $\Vol(j)$ is positive, $\Vol(\rho)$ could very well be negative. Interestingly, if $\rho'$ is the conjugate of the representation $\rho$ by some orientation reversing isometry, then $\Vol(\rho') = -\Vol(\rho)$. In that case, both $\rho$ and $\rho'$ are strictly dominated by $j$, but the quotients $j\times\rho(\Gamma) \backslash \SO_0(n,1)$ and $j\times\rho'(\Gamma) \backslash \SO_0(n,1)$ do not have the same volume (unless $\Vol(\rho) = 0$).

\section{Volume of complete anti-de Sitter $3$-manifolds} \label{s:AdS}

The initial motivation for this work is the case when $n=2$. Indeed, the Killing metric on the Lie group $\SO_0(2,1) = \PSL(2,\R)$ is a Lorentz metric of constant negative sectional curvature, and Gu\'eritaud--Kassel's quotients in that particular case are thus \emph{anti-de Sitter} $3$-manifolds.\\

Conversely, as consequence of a theorem of Klingler \cite{Klingler96}, every closed anti-de Sitter $3$-manifold is a quotient of the universal cover of $\SO_0(2,1)$ by a discrete subgroup of transformations preserving the Killing metric. Kulkarni and Raymond proved that these manifolds are actually quotients of a finite cover of $\SO_0(2,1)$ \cite{KulkarniRaymond85}. Therefore, every closed anti-de Sitter $3$-manifold is, up to a finite cover, one of Gu\'eritaud--Kassel's quotients.

In a recent survey by specialists of this domain \cite{QuestionsAdS}, it is asked whether the volume of these manifolds is rigid (Question 2.3). Our theorem provides an answer. Let us reformulate it in this particular case. Recall that a torsion-free cocompact lattice of $\SO_0(2,1)$ is isomorphic to the fundamental group of some closed orientable surface $S$ of genus $\geq 2$. In that case, the volume of a representation $\rho: \pi_1(S) \to \SO_0(2,1)$ is $2\pi$ times the \emph{Euler class} of the representation, denoted $\euler(\rho)$.

\begin{theo}[Theorem \ref{t:VolumeGKQuotients} in the AdS case] \label{t:VolumeClosedAdS}
Let $S$ be a closed surface of genus at least $2$ an $j$ and $\rho$ two representations of $\pi_1(S)$ into $\SO_0(2,1)$ such that $j$ is discrete and faithful and such that there exists a $(j,\rho)$-equivariant contracting map from $\H^2$ to $\H^2$. Then
\[\Vol\left( j\times \rho(\pi_1(S)) \backslash \SO_0(2,1) \right) = 4 \pi^2 (\euler(j) + \euler(\rho))~.\]
In particular, this volume is an integral multiple of $4\pi^2$ and is therefore constant under continuous deformations of $j$ and $\rho$.
\end{theo}

This theorem was initially proven in the author's thesis (\cite{TholozanThese}, ch.4). There, more details are given about the various possible values for the volume of closed anti-de Sitter $3$-manifolds. In particular, the volume is expressed in terms of the topology of the quotient (which is a Seifert bundle over a hyperbolic orbifold) and an invariant that we called the \emph{length of the fiber}. This expression follows from Theorem \ref{t:VolumeClosedAdS}, but it requires a discussion about how to lift an action on $\SO_0(2,1)$ to a finite cover of $\SO_0(2,1)$. For simplicity, we chose here to restrict to the setting of surface group actions on $\SO_0(2,1)$. Though it only describes closed $\AdS$ $3$-manifolds \emph{up to a finite cover}, this is enough to obtain the volume rigidity of all closed anti-de Sitter manifolds.

Note that the expression of the volume given in Theorem \ref{t:VolumeClosedAdS} differs from the author's thesis by a factor $8$. This is due to the fact that, here, we normalized the Killing metric on $\SO_0(2,1)$ so that it induces a metric of curvature $-1$ on $\H^2$ while, in \cite{TholozanThese}, the Killing metric on $\SO_0(2,1)$ is normalized so that it has itself sectional curvature $-1$. \\

We now turn to non-compact quotients of $\AdS^3$ and prove that their volume is not rigid (Theorem \ref{t:NonRigidity}).

\begin{proof}[Proof of theorem \ref{t:NonRigidity}]
Let $\Gamma$ be a non-cocompact lattice in $\SO_0(2,1)$ and $j:\Gamma \to \SO_0(2,1)$ the inclusion. Then, up to taking a torsion-free finite index subgroup, $j(\Gamma) \backslash \H^2$ is a hyperbolic surface with at least one cusp. It is thus conformal to a $S \backslash \{p_1, \cdots, p_k\}$, where $S$ is a closed Riemann surface and $p_1, \cdots, p_k$ a finite set of points. Let us denote by $g_1$ the conformal hyperbolic metric on $S \backslash \{p_1, \cdots, p_k\}$ with some cusp at each $p_i$.

For every $0 < t < 1$, a famous theorem of Troyanov  \cite{Troyanov91} garanties the existence of a (unique) conformal metric $g_t$ of constant curvature $-1$ on $S$, with conical singularities of angle
\[\alpha(t) = 2 \pi (1-t)\left(1 + \frac{2g-2}{k}\right)\]
at each $p_i$ (here $g$ denotes the genus of $S$).

For each $0<t<1$, the metric $g_t$ induces a holonomy representation \[\rho_t : \Gamma = \pi_1(S \backslash \{p_1, \ldots, p_k\}) \to \SO_0(2,1)\]
and a developing map $f_t: \H^2 \simeq \tilde{S} \to \H^2$ which is $(j,\rho_t)$-equivariant. Moreover, since $g_t$ is conformal to $g_1$, the developing map $f_t$ is holomorphic. By Schwarz's lemma, it is either isometric or locally contracting at every point.

The local form of $g_t$ and $g_1$ at a cusp shows that $g_t/g_1$ goes to $0$ at each $p_i$. Therefore $f_t$ is not an isometry and, by compacity of $S$, it is actually a global contraction. Hence $j \times \rho_t(\Gamma)$ satisfies Gu\'eritaud--Kassel's criterion and we have
\[\Vol \left( j\times\rho_t(\Gamma) \backslash \SO_0(2,1) \right) = 2\pi \left( \Vol(j) + \Vol(\rho_t)\right)~.\]
Since the volume of $\rho_t$ is nothing but the volume of $(S,g_t)$, the Gauss--Bonnet formula gives
\[\Vol(\rho_t) = 2\pi t (2g-2+k) = t \Vol(\Gamma \backslash \H^2)~,\]
and therefore
\[\Vol \left( j\times\rho_t(\Gamma) \backslash \SO_0(2,1) \right) = 2\pi (1+t) \Vol(\Gamma \backslash \H^2)~.\]
When $t$ goes to $0$, $\rho_t$ converges (up to conjugation) to an elliptic representation $\rho_0$ (i.e. a representation fixing a point in $\H^2$). Finally, one can define $\rho_{-t}$ as the representation $\rho_t$ conjugated by an orientation reversing isometry. This way, we have
\[\Vol(\rho_t) = t \Vol(\Gamma \backslash \H^2)\]
for all $-1 < t < 1$.
Each $j\times \rho_t(\Gamma)$ satisfies Gu\'eritaud--Kassel's criterion and the volume of $j\times\rho_t(\Gamma) \backslash \SO_0(2,1)$ varies linearly between $0$ and $2 \pi \Vol(\Gamma \backslash \H^2)$.
\end{proof}

\section{Rigidity of the volume in higher dimension} \label{s:Rigidity}

The purpose of this section is to show that the non-rigidity phenomenon that we just exhibited cannot happen in higher dimension. According to Theorem \ref{t:VolumeGKQuotients}, it is enough to prove the following:

\begin{lem} \label{l:RigidityVolRep}
Let $\Gamma$ be a (not necessarily cocompact) lattice in $\SO_0(n,1)$, $n\geq 3$. Denote by $j$ the inclusion of $\Gamma$ into $\SO_0(n,1)$ and let $\rho_t$ be a continuous family of representations of $\Gamma$ into $\SO_0(n,1)$ that are strictly dominated by $j$. Then
\[ t \mapsto \Vol(\rho_t)\]
is constant.
\end{lem}

Actually, the assumption that $\rho_t$ is strictly dominated by $j$ is only needed when $n=3$ and $\Gamma$ is not cocompact. When $\Gamma$ is cocompact, the result is due to Besson--Courtois--Gallot \cite{BCG07}. It was extended by Kim and Kim for representations of non-cocompact lattices of $\SO_0(n,1)$ when $n\geq 4$ \cite{KimKim13}.

The most problematic case is when $n=3$ and $\Gamma$ is non-compact. Here the lemma is not true anymore if we remove the hypothesis that $j$ dominates $\rho_t$. In that case, lemma \ref{l:RigidityVolRep} is (and the proof will be) very similar to \cite[Theorem 1.3]{KimKim13}. Our first step will be to give a precise definition of the volume of a representation.

\subsection{Volume of representations of a lattice} \label{ss:VolumeRepresentation}

Let $\Gamma$ be a lattice in $\SO_0(n,1)$, $n\geq 3$. Denote by $j$ the inclusion of $\Gamma$ into $\SO_0(n,1)$. Following \cite{GueritaudKassel}, we call a representation $\rho:\Gamma \to \SO_0(n,1)$ \emph{cusp preserving} (resp. \emph{cusp deteriorating}) if for every parabolic element $\gamma \in \Gamma$, $\rho(\gamma)$ is parabolic or elliptic (resp. elliptic). Note that, ultimately, the only representations we are interested in here are cusp deteriorating, according to the following proposition:

\begin{prop}[\cite{GueritaudKassel}, Lemma 2.6]
If $\rho$ is strictly dominated by $j$, then $\rho$ is cusp deteriorating.
\end{prop}

If $\rho$ is strictly dominated by $j$, theorem \ref{t:VolumeGKQuotients} is true if we define the volume of $\rho$ as
\[\int_{j(\Gamma) \backslash \H^n} f^*\vol^{\H^n}~,\]
where $f$ is any $(j,\rho)$-equivariant contracting map from $\H^n$ to $\H^n$. According to corollary \ref{c:FiniteVolume}, this value does not depend on the choice of a contracting map. More generally, let us prove the following:

\begin{prop} \label{p:VolumeLipschitz}
If $f_0$ and $f_1$ are two $(j,\rho)$-equivariant maps that are $C$-Lipschitz, then
\[\int_{j(\Gamma) \backslash \H^n} f_0^*\vol^{\H^n} = \int_{j(\Gamma) \backslash \H^n} f_1^*\vol^{\H^n}~.\]
\end{prop}

\begin{rmk}
The existence of a $(j,\rho)$-equivariant Lipschitz map is equivalent to $\rho$ being cusp-preserving (\cite[Lemma 4.9]{GueritaudKassel}).
\end{rmk}

\begin{proof}
When $\Gamma$ is cocompact, the fact that $\int_{j(\Gamma) \backslash \H^n} f^*\vol^{\H^n}$ does not depend on $f$ is a classical consequence of Stokes's formula. In the non-compact case we simply apply Stokes formula to the complement of some cusped regions. A boundary term appears, and we must prove that this term goes to $0$ when these cusped regions are chosen smaller and smaller. \\

Let $f_t$ be the map from $\H^n$ to $\H^n$ sending a point $x$ to the barycenter
\[\Bar \left\{ (f_0(x), 1-t);(f_1(x),t) \right\}\]
(i.e. the point on the segment $[f_0(x), f_1(x)]$ satisfying \[d(f_t(x),f_0(x)) = t\, d(f_1(x), f_0(x))~).\]
By convexity of the hyperbolic distance, $f_t$ is $C$-Lipschitz for all $t$. By equivariance of $f_0$ and $f_1$, each $f_t$ is $(j,\rho)$-equivariant. Putting all these maps together, we obtain a $(j,\rho)$-equivariant map $F: \H^n\times [0,1] \to \H^n$.

Denote $M = j(\Gamma) \backslash \H^n$. A \emph{cusp} in $M$ is an open connect domain whose lift to $\H^n$ is a disjoint union of horoballs. Let $K_0$ be a compact in $M$ whose complement is a disjoint union of cusps. Let $K_p$ be the set of points of $M$ at distance at most $p$ from $K_0$. Then the complement of $K_p$ in $M$ is still a disjoint union of cusps.

Since $\vol^{\H^n}$ is closed, by Stokes's formula, we have
\[\int_{\partial (K_p\times [0,1])} F^*\vol^{\H^n} = 0~,\]
hence
\begin{equation} \label{eq:Stokes}
\int_{K_p} f_0^*\vol^{\H^n} - \int_{K_p} f_1^*\vol^{\H^n} = \int_{[0,1] \times \partial K_p} F^*\vol^{\H^n}~.
\end{equation}
We need to prove that the right hand term goes to $0$ when $p$ goes to infinity.

Since each $f_t$ is $C$-Lipschitz and since $F(x, \cdot)$ maps $[0,1]$ to the segment $[f_0(x), f_1(x)]$, it follows that
\[ \left| \int_{[0,1] \times \partial K_p} F^*\vol^{\H^n} \right | \leq C^{n-1}\times \area(\partial K_p) \times \max_{x\in \partial \tilde{K}_p} d(f_0(x), f_1(x))~.\]
By definition of $K_p$, we have
\[\max_{x\in \partial \tilde{K}_p} d(f_0(x), f_1(x)) \leq \max_{x\in \partial \tilde{K}_0} d(f_0(x), f_1(x)) + 2 Cp~.\]
On the other side, it is a well-known fact in hyperbolic geometry that
\[\area(\partial K_p) = e^{-(n-1)p} \area(\partial K_0)~.\]
Putting these two facts together we obtain
\begin{eqnarray*}
\left| \int_{[0,1] \times \partial K_n} F^*\vol^{\H^n} \right| & \leq & (\Cste\, p + \Cste') e^{-(n-1)p}\\
\ & \tend{n\to +\infty} & 0~.
\end{eqnarray*}
By equation \eqref{eq:Stokes}, we finally obtain
\begin{eqnarray*}
\int_M f_0^*\vol^{\H^n} & = & \lim_{n\to +\infty} \int_{K_n} f_0^*\vol^{\H^n} \\
\ & = & \lim_{n\to +\infty} \int_{K_n} f_1^*\vol^{\H^n} \\
\ & = & \int_M f_0^*\vol^{\H^n}~.
\end{eqnarray*}

\end{proof}

Proposition \ref{p:VolumeLipschitz} allows us to define the volume of a cusp preserving representation $\rho$ as the integral of the pull-back volume form by any $(j,\rho)$-equivariant Lipschitz map. We will show that this is a particular case of a more general definition.

Let ${\widehat{\H}^n}_\Gamma$ denote the union of $\H^n$ with all the fixed points in $\partial_\infty \H^n$ of parabolic elements in $\Gamma$. We provide ${\widehat{\H}^n}_\Gamma$ with the topology extending the topology of $\H^n$ an such that for any point $x\in {\widehat{\H}^n}_\Gamma \cap \partial_\infty \H^n$, a neighbourhood basis of $x$ is given by the family of horoballs tangent to $x$. With this topology, $j(\Gamma)$ acts properly discontinuously and cocompactly on ${\widehat{\H}^n}_\Gamma$, and the quotient is the compactification of $j(\Gamma) \backslash \H^n$ where one point as been ``added'' at the end of each cusp. 

\begin{defi}
A piecewise smooth $(j,\rho)$-equivariant map $f:\H^n \to \H^n$ is \emph{nicely ending} if it extends continuously to a map from ${\widehat{\H}^n}_\Gamma$ to $\overline{\H}^n$.
\end{defi}

\begin{propdef}[Kim--Kim, \cite{KimKim13}] \label{d:VolumeRepresentation}
Let $\rho: \Gamma \to \SO_0(n,1)$ be a representation. Let $f:\H^n \to \H^n$ be a $(j,\rho)$-equivariant nicely ending map. Then
\[\int_{j(\Gamma)\backslash \H^n} f^*\vol^{\H^n}\]
is finite and does not depend on the choice of $f$. This number is called the \emph{volume} of the representation $\rho$.
\end{propdef}

A particular case of a nicely ending map $f$ is a so called \emph{pseudo-developing} or \emph{properly ending} map (\cite{Francaviglia04}, \cite{FrancavigliaKlaff06}). Therefore, this definition of $\Vol(\rho)$ coincides with the definition of the volume of $\rho$ by Dunfield \cite{Dunfield99} and Francaviglia \cite{Francaviglia04}. Kim and Kim show that one can always construct a nicely ending \emph{simplicial} map $f$ and decude that this definition of $\Vol(\rho)$ also coincides with the definition by Bucher--Burger--Iozzi using relative cohomology \cite{BBI13}. Here we will prove that it also coincides with our definition:

\begin{prop}
Let $\rho$ be a cusp preserving representation of $\Gamma$ into $\SO_0(n,1)$. Then for any $(j,\rho)$-equivariant Lipschitz map $f$, 
\[\int_{j(\Gamma)\backslash \H^n} f^*\vol^{\H^n}\]
is equal to the volume of $\rho$ (in the sense of Kim and Kim).
\end{prop}

\begin{proof}
by propostion \ref{p:VolumeLipschitz} and proposition \ref{d:VolumeRepresentation}, it is enough to find a nicely ending Lipschitz map. We start by defining this map in the (lifts of the) cusps. Since the complement is compact, it is then easy to extend it as a global Lipschitz map.\\

Let us choose a compact core in $M= j(\Gamma) \backslash \H^n$ whose lift to $\H^n$ is the complement of a disjoint union of horoballs. The action of $j(\Gamma)$ on the space of these horoballs has finitely many orbits. Let us choose $B_1,\ldots, B_k$ some representatives in each orbit, and denote by $\Gamma_i$ the subgroup of $\Gamma$ whose image by $j$ stabilizes $B_i$.

The group $j(\Gamma_i)$ contains no hyperbolic element. Since $\rho$ is cusp preserving, this is also true for $\rho(\Gamma_i)$. It is then a classical fact that one of the following holds:
\begin{itemize}
\item either $\rho(\Gamma_i)$ fixes a point $y_i$ in $\H^n$
\item or $\rho(\Gamma_i)$ fixes a point $y_i$ in $\partial_\infty \H^n$ and preserves a Buseman function $\beta_{y_i}$ associated to $y_i$
\end{itemize}
In the first case, we say that $\rho(\Gamma_i)$ is of elliptic type. In the second case, we say that $\rho(\Gamma_i)$ is of parabolic type.

If $\rho(\Gamma_i)$ is of parabolic type, then it commutes with the flow $F_{y_i}$ such that $F_{y_i}(x,t)$ is at distance $t$ from $x$ on the geodesic ray from $x$ to $y_i$. Let $x_i$ be the point in $\partial_\infty \H^n$ fixed by $j(\Gamma_i)$. One can find a $(\Gamma_i, j,\rho)$-equivariant Lipschitz map $f$ from the horosphere $\partial B_i$ to the horosphere $\{\beta_{y_i} = 0\}$, and then extend it to $B_i$ by setting
\[f(F_{x_i}(x,t)) = F_{y_i}(f(x),t)~.\]
One can check that $f$ is still Lipschitz.

If $\rho(\Gamma_i)$ is of elliptic type, one can simply set $f(B_i) = \{y_i\}$. There is now a unique way of extending $f$ equivariantly to
\[ \bigsqcup_{i= 1}^k \bigsqcup_{\gamma \in \Gamma} j(\gamma) \cdot B_i~.\]
Finally, one can extend $f$ smoothly and equivariantly to $\H^n$. Since $f$ is locally $C$-Lipschitz on $B_i$ (for some constant $C$), and since $j(\Gamma) \backslash \left( \H^n -\bigsqcup_{i=1}^n \Gamma \cdot B_i \right)$ is compact, the resulting map $f$ is globally Lipschitz.

Finally, by construction, $f$ extends naturally to a map from ${\widehat{\H}^n}_\Gamma$ to $\overline{\H}^n$ by setting
\[f(x_i) = y_i~.\]
Therefore, $f$ is a nicely ending Lipschitz map.
\end{proof}

\subsection{Proof of lemma \ref{l:RigidityVolRep}}

The main ideas of the proof are due to Besson--Courtois--Gallot \cite{BCG07} and Kim--Kim \cite{KimKim13}. Here we will only sketch the general strategy and explain how to overcome the technical difficulties arising when working with non-compact lattices.\\

 Besson--Courtois--Gallot's idea to prove the volume rigidity of representations in $\SO_0(n,1)$ is to use the Schl\"afli formula. Let us start by recalling this formula.

Let $x_0, \ldots, x_n$ be $n+1$ points in $\H^n$. We say that $(x_0,\ldots, x_n)$ are (projectively) independent if their convex hull as non-empty interior. In that case, the convex hull is an $n$-dimensional tetrahedron denoted $[x_0,\ldots, x_n]$. For any pairwise distinct $i_0, \ldots, i_k$ in $\{0,\ldots, n\}$, the convex hull of $x_{i_0}, \ldots, x_{i_k}$ forms a $k$-face of $[x_0,\ldots, x_n]$, denoted $[x_{i_0}, \ldots, x_{i_k}]$.

\begin{prop}[Schl\"afli formula] \label{p:SchlafliFormula}
Let $x_0(t), \ldots, x_n(t)$ be $\CC^1$ curves in $\H^n$ that are projectively independent at each time $t$. Then
\[\dt \Vol[x_0, \ldots, x_n] = \sum_{i < j} \area[x_0, \ldots, \hat{x}_i, \hat{x}_j, \ldots, x_n]\cdot \dt \theta(x_i,x_j)~,\]
where $\theta(x_i,x_j)$ denotes the dihedral angle at the $(n-2)$-face $[x_0, \ldots, \hat{x}_i, \hat{x}_j, \ldots, x_n]$.
\end{prop}

Besson--Courtois--Gallot use this to prove the rigidity of the volume of a representation of a cocompact lattice. Indeed, in the cocompact case, one can fix a triangulation of $M = j(\Gamma) \backslash \H^n$ and choose equivariant maps that are simplicial (i.e. cells of the triangulation are mapped to hyperbolic tetrathedra). The dihedral angles around a given $(n-2)$-cell then add up to a multiple of $2\pi$. This implies a cancellation when computing the derivative of the volume along a $\CC^1$ family of representations.

Kim and Kim have adapted this proof to the non-cocompact case by considering a triangulation of $M$ with an ``ideal'' vertex at the end of each cusp and allowing to map these ideal vertices to the boundary of $\H^n$. This gives the rigidity in dimension $n\geq 4$, where the Schl\"afli formula naturally extends to tetrahedra with ideal vertices, but it partly fails in dimension $3$ because, in that case, the Schl\"affli formula does not make sense when some vertex is at infinity (indeed, the area of an $(n-2)$-face is the length of an edge and becomes infinite when some vertex is in the boundary). However, using a variation of the Schl\"afli formula in that case, they were able to prove the rigidity of $\Vol(\rho)$ in restriction to representations sending parabolic elements to parabolic elements.

The situation is simpler when we restrict to representations $\rho$ that are strictly dominated by $j$. Indeed, such representations are cusp deteriorating and, therefore, the image of each cusp group is of elliptic type and fixes a point in $\H^n$. One can choose a simplicial map sending the vertex at the end of a cusp to such a fixed point. This way, the image of each cell is a compact tetrahedron in $\H^n$. One can hence use the Schl\"afli formula and Besson--Courtois--Gallot's argument to prove that the volume is constant along every $\CC^1$ path in the space $\Dom(\Gamma, \SO_0(n,1))$ of representations of $\Gamma$ into $\SO_0(n,1)$ that are strictly dominated by $j$.

Now, $\Dom(\Gamma, \SO_0(n,1))$ is an open set in a real algebraic variety. Therefore, the space of smooth points is dense in $\Dom(\Gamma, \SO_0(n,1))$ and we conclude that the volume function is locally constant on $\Dom(\Gamma, \SO_0(n,1))$.

\section{Computation of the volume} \label{s:ComputationVolume}

This section is devoted to the proof of theorem \ref{t:VolumeGKQuotients}, and more precisely to theorem \ref{t:MainThmBis}. From now on, we fix an integer $n \geq 2$, a torsion-free finitely generated subgroup $\Gamma$, a discrete and faithful representation $j: \Gamma \to \SO_0(n,1)$ and another representation $\rho: \Gamma \to \SO_0(n,1)$ such that there is a map $f: \H^n \to \H^n$ which is $(j,\rho)$-equivariant and contracting.

\subsection{Bundle structure of Gu\'eritaud--Kassel's quotients} \label{ss:DescriptionGKQuotients}

The map $f$ allows to describe the quotient $(j\times \rho)(\Gamma) \backslash \SO_0(n,1)$ as a bundle over $j(\Gamma) \backslash \H^n$ whose fibers have the form $K g$, with $g$ an element of $\SO_0(n,1)$ and $K$ a maximal compact subgroup of $\SO_0(n,1)$ (i.e. a subgroup conjugated to $\SO(n)$). The way we do this construction here is a differentiable version of the construction given in \cite{GueritaudKassel}. \\

We can assume without loss of generality that $f$ is smooth. Let us introduce the map
\[\pi  : \SO_0(n,1) \to \H^n\]
such that $\pi(g)$ is the unique fixed point of $g\circ f$. This map is well defined because $f$ (and hence $g\circ f$) is a contracting map.

The following propositions follow easily from the definition.
\begin{prop} \label{p:FibersPi}
The map $\pi$ is onto and for all $x \in \H^n$, one has
\[\pi^{-1}(x) = K_x g~,\]
where $g$ is some element in $\pi^{-1}(x)$ and $K_x$ is the stabilizer of $x$.
\end{prop}

\begin{prop} \label{p:EquivariancePi}
The map $\pi$ is $\Gamma$-equivariant, i.e. for all $\gamma \in \Gamma$ and all $g \in \SO_0(n,1)$,
\[\pi\left( j(\gamma) g \rho(\gamma)^{-1}\right) = j(\gamma) \cdot \pi(g)~.\]
\end{prop}

\begin{proof}[Proof of proposition \ref{p:FibersPi}]
By transitivity of the action of $\SO(n,1)$ on $\H^n$, there exists, for all $x \in \H^n$, some $g\in \SO(n,1)$ such that $g\cdot f(x) = x$. We thus have $\pi(g) = x$. Hence $\pi$ is onto. Let $h$ be another element in $\pi^{-1}(x)$. Then
\[hg^{-1} \cdot x = hg^{-1} g \cdot f(x) = h \cdot f(x) = x~.\]
Hence $h\in K_x g$. The same computation shows that, conversely, if $h\in K_x g$ then $\pi(h) = x$.
\end{proof}

\begin{proof}[Proof of proposition \ref{p:EquivariancePi}]
Fix some $\gamma \in \Gamma$ and some $g\in \SO(n,1)$. One easily verifies that $j(\gamma) g \rho(\gamma)^{-1} \circ f$ fixes $j(\gamma) \cdot \pi(g)$. Indeed,
\begin{displaymath}
\begin{array}{rcll}
j(\gamma) g \rho(\gamma)^{-1} \cdot f\left( j(\gamma) \cdot \pi(g)\right) &  = & j(\gamma) g \cdot f(\pi(g)) & \textrm{(by equivariance of $f$)}\\
\ & = & j(\gamma) \cdot \pi(g) & \textrm{(by definition of $\pi$).}
\end{array}
\end{displaymath}
Therefore
\[\pi(j(\gamma) g \rho(\gamma)^{-1}) = j(\gamma) \cdot \pi(g)~.\]
\end{proof}

Finally , let us compute the first derivative of $\pi$. This will show that $\pi$ is a submersion and hence a smooth fibration.

\begin{lem} \label{l:DifferentialPi}
Let $g$ be a point in $\SO_0(n,1)$. Identify $T_g\SO_0(n,1)$ with the Lie algebra of right invariant vector fields on $\SO_0(n,1)$. Given $u$ such a vector field, denote by $X_u$ the induced vector field on $\H^n$. Then we have
\[\d_g \pi(u) = \left( \Id - \d_{\pi(g)} (g\circ f)\right)^{-1} \left(X_u(\pi(g)) \right)~.\]
\end{lem}

The proof will rely on the following application of the implicit function theorem:

\begin{prop} \label{l:PtFixeFonctionsImplicites}
Let $U$ be a neighbouhood of $0$ in $\R^n$ and $f: U \to \R^n$ a map such that $f(0) = 0$ and such that $\normop[\d f_0] < 1$ for some operator norm on $\R^n$. Let $X$ be a vector field on $U$, with flow $\Phi_t$ defined on $]-\epsilon, \epsilon[ \times V$ for some neighbourhood $V$ of $0$ and some $\epsilon >0$. 

Then, up to taking some smaller $\epsilon$, $\Phi_t \circ f$ has a unique fixed point $x(t)$ for all $t \in ]-\epsilon, \epsilon[$. Moreover, the function $t \mapsto x(t)$ is differentiable at $0$ and we have
\[ \dot{x}(0)= \left(\Id - \d f_0\right)^{-1}(X_0)~.\]
\end{prop}

\begin{proof}[Proof of proposition \ref{l:PtFixeFonctionsImplicites}]
Define
\[\function{F}{]-\epsilon, \epsilon[\times V}{\R^d}{(t,x)}{\Phi_t\circ f(x)-x~.}\]
We have
\[ \d_{(0,0)} F{(s,u)} = s X_0 + \d_0 f(u) - u~.\]
Since $\norm[\d f_0] < 1$, the linear map $\d F_{(0,0)}$ restricted to $\{0\}\times \R^d$ is bijective. Therefore $F$ is a submersion in a neighbouhood of $(0,0)$ and $F^{-1}(0)$ is thus a submanifold of dimension $1$, transverse to the fibers $t= \Cste$. The locus $F^{-1}(0)$ can thus be written locally as $\{(t,x(t)), t\in ]-\epsilon, \epsilon[\}$, where $x(t)$ is smooth and satisfies
\[X_0 + \d f_0(\dot{x}(0)) - \dot{x}(0) = 0~.\]
The conclusion follows.
\end{proof}

\begin{proof}[Proof of lemma \ref{l:DifferentialPi}]
Let us now compute the differential of $\pi$. Fix $g \in \SO_0(n,1)$ and $u$ a vector in the Lie algebra of $\SO_0(n,1)$. We identify $u$ with a tangent vector at $g$ by right translation, i.e. we also denote $u$ the vector
\[\dt_{|t=0} \exp(tu)g~.\]
Let $X_u$ be the vector field on $\H^n$ associated to $u$, i.e. the generator of the flow
\[\Phi_t(x) = \exp(tu) \cdot x~.\]

By definition, $\pi(g)$ is fixed by $g\circ f$ and, since $f$ is contracting, one has $\norm[\d_{\pi(g)} (g\circ f) ] < 1$. Applying lemma \ref{l:PtFixeFonctionsImplicites} to $g\circ f$ in a neighbourhood of $\pi(g)$ and to the vector field $X_u$, we deduce that $t \mapsto \pi(\exp(tu) g)$ is differentiable at $0$ and
\[\d_g \pi(u) = \dt_{|t=0} \pi(\exp(tu)g) = \left( \Id - \d_{\pi(g)} (g\circ f)\right)^{-1} \left(X_u(\pi(g)) \right)~.\]

Note that the map $u \mapsto X_u(\pi(g))$ is a surjective map from $\so(n,1)$ to $T_{\pi(g)}\H^n$ since the action of $\SO_0(n,1)$ on $\H^n$ is transitive. Moreover, the linear map $\left(\Id - \d_{\pi(g)} (g\circ f)\right)^{-1}$ is bijective. Hence $\pi$ is a submersion at $g$.

\end{proof}

In conclusion, the map $\pi$ induces a smooth fibration from $j\times\rho(\Gamma) \backslash \SO_0(n,1)$ to $j(\Gamma) \backslash \H^n$, whose fibers are images of $\SO(n)$ by left and right multiplication. With a slight abuse of notation, we will still denote by $\pi$ the induced fibration.

\subsection{Integration of the volume along the fibers}

In order to compute the volume of $j\times\rho(\Gamma) \backslash \SO_0(n,1)$, we will first integrate the volume form along the fibers of $\pi$.\\

Let $g$ be a point in $\SO_0(n,1)$. The Lie algebra $\so(n,1)$ decomposes as $\k_g \oplus \k_g^\perp$, where $\k_g$ is the Lie algebra of the stabilizer of $\pi(g)$ and $\k_g^\perp$ its orthogonal with respect to the Killing form. Recall that this Lie algebra is identified to the tangent space $T_g\SO_0(n,1)$ by right multiplication.

By lemma \ref{l:DifferentialPi}, the kernel of the differential of $\pi$ at $g$ is precisely $\k_g$, and $\d_g \pi$ thus induces an isomorphism from $\k_g^\perp$ to $T_{\pi(g)} \H^n$. One can thus define a volume form $\vol^{\pi}$ on $T_g\SO_0(n,1)$ by providing $\k_g$ with the volume associated to the restricted Killing metric and $\k_g^\perp$ with the pull-back of the volume form $\vol^{\H^n}$ by $\pi$.

\begin{prop}
The volume form $\vol^\pi$ defined above and the volume form $\vol^\Kill$ associated to the Killing metric are related by
\[\vol^\Kill_g = \det \left(\Id - \d_{\pi(g)} (g \circ f)\right) \vol^\pi_g~.\]
\end{prop}

\begin{proof} 
Let $(u_1, \ldots u_n)$ be an orthonormal basis of $\k_g^\perp$ and $(v_1, \ldots, v_{n(n-1)/2})$ an orthonormal basis of $\k_g$. Then, up to the sign, one has
\[\vol^\kappa(u_1, \ldots u_n, v_1, \ldots, v_{n(n-1)/2}) = 1~.\]
On the other side, the Killing fields $X_i = X_{u_i}$ on $\H^n$ form an orthonormal basis of $T_{\pi(g)} \H^n$ and thus satisfy
\[\vol^{\H^n}_{\pi(g)}(X_1, \ldots, X_n)= 1~.\] 

By lemma \ref{l:DifferentialPi}, we know that
\[\d_g \pi(u_i) = \left(\Id - \d_{\pi(g)} (g \circ f)\right)^{-1}(X_i)~.\]
Therefore,
\[\vol^{\H^n}_{\pi(g)}(\d_g\pi(u_1),\ldots, \d_g \pi(u_n)) = \det\left(\Id - \d_{\pi(g)} (g \circ f)\right)^{-1} \vol^{\H^n}_{\pi(g)}(X_1, \ldots, X_n)~,\]
hence
\[\vol^{\Kill}_g = \det\left(\Id - \d_{\pi(g)} (g \circ f)\right) \vol^\pi_g~.\]
\end{proof}

In order to integrate this identity along the fibers, we will need the following linear algebra lemma:

\begin{lem} \label{l:LinearAlgebra}
Let $\vol^K$ denote the volume form of $\SO(n)$. Then for any matrix $A \in \M_n(\R)$, we have
\[\int_{\SO(n)} \det\left( \I_n - U A \right) \d \vol^K(U) = \V_n \left(1 + (-1)^n \det(A) \right)~.\]
\end{lem}

\begin{proof}
If $B$ is some $n\times n$ matrix and $P$ a subset of $\{1,\ldots, n\}$, let us denote by $B_P$ the square matrix of size $n- |P|$ obtained from $B$ by removing all lines and columns whose indices are in $P$. Then one has
\[\det(\lambda \I_n - B) = \sum_{k=0}^n \left(\sum_{|P| = k} \det(B_P) \right) (-1)^{n-k}\lambda^k~.\]
Therefore,
\begin{equation} \label{eq:PolynomeCaracteristique}
\int_{\SO(n)} \det\left( \I_n - U A \right) \d \vol^K(U) = \sum_{k=0}^n \left(\sum_{|P| = k} \int_{\SO(n)} \det((UA)_P) \d \vol^K(U)~. \right)\end{equation}
Now, for any $P \subset \{1,\ldots, n\}$ with $1 \leq |P| \leq n-1$, let $D$ be the diagonal matrix with coefficients $(d_1, \ldots, d_n)$ where $d_i$ is $-1$ if $i = \inf P$ or $i = \inf \{1,\ldots, n\} \backslash P$ and $d_i = 1$ otherwise. Then $D$ is in $\SO(n)$ and, by construction, 
\[\det((DB)_P) = - \det(B_P)\]
for all matrix $B$. (Indeed, the multiplication by $D$ multiplies the first line of $B_P$ by $-1$.)

We thus have
\begin{eqnarray*}
\int_{\SO(n)} \det((UA)_P) \d \vol^K(U) & = & \int_{\SO(n)} \det((DUA)_P) \d \vol^K(U) \\
\ & = & \int_{\SO(n)} -\det((UA)_P) \d \vol^K(U)~,
\end{eqnarray*}
hence
\[\int_{\SO(n)} \det((UA)_P) \d \vol^K(U) = 0~.\]
The only non vanishing terms in the sum \eqref{eq:PolynomeCaracteristique} thus appear when $P= \emptyset$ or $P = \{1,\ldots, n\}$. Therefore,
\begin{eqnarray*}
\int_{\SO(n)} \det\left( \I_n - U A \right) \d \vol^K(U) & = & \int_{\SO(n)} \d \vol^K(U) \\
\ & \ & + (-1)^n \int_{\SO(n)} \det(UA) \d \vol^K(U) \\
\ & = & (1+(-1)^n \det(A)) \int_{\SO(n)} \d \vol^K(U) \\
\ & = & \V_n \left(1 + (-1)^n \det(A) \right)~.
\end{eqnarray*}

\end{proof}

We can now complete the proof of theorem \ref{t:VolumeGKQuotients}. Let us denote by $B$ the manifold $j(\Gamma) \backslash \H^n$ and by $E$ the manifold $(j\times \rho)(\Gamma) \backslash \SO_0(n,1)$. Let $U$ be an open set of $B$ over which the fibration $\pi: E\to B$ admits a section $s$. One can then identify the open set $V= \pi^{-1}(U)$ with $K \times U$, where $K$ is the stabilizer of a base point $x_0$ in $U$. Such a trivialization is given by
\[(k,x) \mapsto hkh^{-1} s(x),\]
where $h$ is any element of $\SO_0(n,1)$ sending $x_0$ to $x$ (so that $hKh^{-1}$ is the stabilizer of $x$). Through this trivialization, the volume $\vol^\pi$ identifies with the product volume on $K \times U$ (where $K \simeq \SO(n)$ is provided with the volume $\vol^K$). On the other side, we have
\[\vol^\Kill_{(k,x)} = \det\left(\Id - \d_x (hkh^{-1}s(x) \circ f)\right) \vol^\pi_{(k,x)}~.\]

When $k$ spans $K$, $hkh^{-1}$ spans $\SO(T_x\H^n)$. By lemma \ref{l:LinearAlgebra}, we thus obtain

\begin{eqnarray*}
\int_K \det\left(\Id - \d_x (hkh^{-1}s(x) \circ f)\right) \vol^K &=& \V_n \left(1 + (-1)^n \det(\d_x s(x)\circ f) \right) \\
\ & = & \V_n \left(1 + (-1)^n \Jac_f(x) \right)~,
\end{eqnarray*}
where $\Jac_f(x)$ is defined by
\[f^*\left(\vol^{\H^n}_{f(x)}\right) = \Jac_f(x) \vol^{\H^n}_x~.\]

Therefore, 
\[\int_{K\times U} \vol^{\Kill} = \int_U \V_n \left( 1 + (-1)^n \Jac_f(x)\right) \vol^{\H^n}~.\]
We finally obtain
\begin{eqnarray*}
\int_E \vol^{\Kill} & = & \int_B \V_n \left(1 + (-1)^n \Jac_f(x)\right)\vol^{\H^n}\\
\  & = & \V_n \int_B \vol^{\H^n} + (-1)^n f^*\left( \vol^{\H^n}\right)~.
\end{eqnarray*}

This concludes the proof of theorem \ref{t:MainThmBis}.

\section{Some open questions} \label{s:OpenQuestions}

In this last section, we discuss further some questions raised by this work.

\subsection{Are finite covolume discrete groups geometrically finite ?} \label{sss:FinitelyGenerated}

It is a famous result of Kazhdan and Ragunathan that lattices of a Lie group are finitely presented. However, to our knowledge, there is no answer to the following more general question:

\begin{ques}
Let $G$ be a Lie group and $X$ a $G$-homogeneous space such that $G$ preserves a volume form on $X$. Let $\Gamma$ be a discrete subgroup of $G$ acting properly discontinuously on $X$ such that $\Gamma \backslash X$ has finite volume. Is $\Gamma$ finitely generated?
\end{ques}

We expect this to be true for quotients of $\SO_0(n,1)$, but we fell short of proving it here. The main issue is that Gu\'eritaud--Kassel's theorem \emph{assumes} that $\Gamma$ is a geometrically finite subgroup of $\SO_0(n,1)$ (in particular it is finitely generated). 

If we remove this assumption, there are examples of proper actions on $\SO_0(n,1)$ given by $j\times \rho(\Gamma)$, where $j$ is discrete and faithful and for which there is a $(j,\rho)$-equivariant map $f:\H^n \to \H^n$ satisfying
\[d(f(x), f(y)) < d(x,y)\]
for all $x\neq y$, but that is not $\lambda$-Lipschitz for some $\lambda < 1$.

Conversely, if such a map exists, then $j\times \rho(\Gamma)$ acts properly discontinuously on the domain $U$ formed by all the $g\in \SO_0(n,1)$ such that $g\circ f$ has a fixed point. The integral
\[\int_{j(\Gamma) \backslash \H^n} \vol^{\H^n} + (-1)^n f^*\vol^{\H^n}\]
then computes the volume of $j\times \rho(\Gamma) \backslash U$. Whether this domain $U$ is the whole hyperbolic space $\H^n$ depends on how contracting the map $f$ is.

Interestingly, when $n=2$, it can happen that $f$ is close enough to being an orientation reversing isometry asymptotically, so that $\vol^{\H^2} + f^*\vol^{\H^2}$ is integrable and thus $j\times \rho(\Gamma) \backslash U$ has finite volume. However, we expect this to imply that $U \varsubsetneq \SO_0(2,1)$. More generally, we conjecture:

\begin{conj}
Let $\Gamma$ be a discrete group and $j\times \rho$ a faithful representation of $\Gamma$ into $\SO_0(n,1)\times \SO_0(n,1)$ such that $j\times \rho(\Gamma)$ acts properly discontinuously on $\SO_0(n,1)$ with finite covolume. Then, up to switching the factors, $j$ is discrete and faithful and $j(\Gamma)$ is a lattice in $\SO_0(n,1)$ (in particular it is geometrically finite).
\end{conj}

\subsection{What are the values of the volume of Gu\'eritaud--Kassel's quotients?}

Let us fix a torsion-free lattice $\Gamma$ in $\SO_0(n,1)$. It is natural to ask what values can be taken by the volume of quotients of $\SO_0(n,1)$ by some action of $\Gamma$. Given Theorem \ref{t:VolumeGKQuotients}, this boils down to understanding the possible values for the volume of a representation $\rho: \Gamma \to \SO_0(n,1)$ which is strictly dominated by a discrete and faithful representation.\\

The case $n=2$ is well-understood. It $\Gamma$ is not cocompact, Theorem \ref{t:NonRigidity} shows that quotients of $\SO_0(2,1)$ can have any volume in the open interval $(0, 2\pi \Vol(j))$, where $j$ is a Fuchsian representation of $\Gamma$. On the contrary, if $\Gamma$ is cocompact, then the volume of a representation $\rho$ is an integral multiple of $2\pi$. If, moreover, $\rho$ is strictly dominated by $j$, then one has $|Vol(\rho)| < \Vol(j)$ (note that the weak inequality is true without any hypothesis on $\rho$ by Milnor--Wood inequality). It follows that
\[\frac{1}{2\pi} \Vol\left(j\times \rho(\Gamma) \backslash \SO_0(2,1) \right) \in  \{1, \ldots, 4g-5\}~,\]
where $g$ is the genus of the surface $j(\Gamma) \backslash \H^2$.

Conversely, Salein \cite{Salein00} proved the existence of pairs $(j,\rho)$ with $j$ Fuchsian and $\rho$ strictly dominated by $j$ of any non-extremal Euler class. Therefore, $\frac{1}{2\pi} \Vol\left(j\times \rho(\Gamma) \backslash \SO_0(2,1) \right)$ can take all the integral values between $1$ and $4g-5$.\\

The picture is much more blury in higher dimension. When $n$ is even, it is known that twice the volume of a representation of $\Gamma$ is an integral multiple of the volume of the sphere of dimension $n$. (When $\Gamma$ is cocompact, it follows from the Chern--Gauss--Bonnet theorem. It was generalized to any lattice by Bucher--Burger--Iozzi \cite{BBI14}.) However, it is not known which of this values are actually realized, let alone by representations strictly dominated by a Fuchsian one. In odd dimension, the only thing that is known is that there are only finitely many possible values, since the space of cusp deteriorating representations has finitely many connected components.\\

So far, the only known examples of Gu\'eritaud--Kassel's quotients when $n \geq 3$ are constructed by taking the representation $\rho$ to be a continuous deformation of the identity. In particular, such a representation has volume~$0$. This raises the following question:

\begin{ques}
Let $\Gamma$ be a lattice in $\SO_0(n,1)$, $n\geq 3$ and $j:\Gamma \to \SO_0(n,1)$ the inclusion. Does there exist a representation $\rho$ strictly dominated by $j$ whose volume is not zero?
\end{ques}

\bibliographystyle{plain}
\bibliography{biblio}

\begin{thebibliography}{10}

\bibitem{AlessandriniLi15}
Daniele Alessandrini and Qiongling Li.
\newblock {A}d{S} 3-manifolds and {H}iggs bundles.
\newblock in preparation.

\bibitem{QuestionsAdS}
Thierry Barbot, Francesco Bonsante, Jeffrey Danciger, William~M. Goldman,
  Fran{\c c}ois Gu\'eritaud, Fanny Kassel, Kirill Krasnov, Jean-Marc Schlenker,
  and Abdelghani Zeghib.
\newblock Some open questions in anti-de {S}itter geometry.
\newblock 2012.
\newblock arXiv:1205.6103.

\bibitem{BCG07}
G{\'e}rard Besson, Gilles Courtois, and Sylvestre Gallot.
\newblock In\'egalit\'es de {M}ilnor--{W}ood g\'eom\'etriques.
\newblock {\em Comment. Math. Helv.}, 82:753–803, 2007.

\bibitem{BBI13}
Michelle Bucher, Marc Burger, and Alessandra Iozzi.
\newblock A dual interpretation of the {G}romov--{T}hurston proof of {M}ostow
  rigidity and volume rigidity for representations of hyperbolic lattices.
\newblock In {\em Trends in Harmonic Analysis}, pages 47--76. Springer, 2013.

\bibitem{BBI14}
Michelle Bucher, Marc Burger, and Alessandra Iozzi.
\newblock Integrality of volumes of representations.
\newblock 2014.
\newblock arXiv:1407.0562.

\bibitem{Dunfield99}
N.~M. Dunfield.
\newblock Cyclic surgery, degrees of maps of character curves, and volume
  rigidity for hyperbolic manifolds.
\newblock {\em Invent. Math.}, 136(3):623--657, 1999.

\bibitem{Francaviglia04}
Stefano Francaviglia.
\newblock Hyperbolic volume of representations of fundamental groups of cusped
  3-manifolds.
\newblock {\em Int. Math. Res. Not.}, 2004(9):425--459, 2004.

\bibitem{FrancavigliaKlaff06}
Stefano Francaviglia and Ben Klaff.
\newblock Maximal volume representations are fuchsian.
\newblock {\em Geometriae Dedicata}, 117(1):111--124, 2006.

\bibitem{Goldman85}
William~M. Goldman.
\newblock Nonstandard {L}orentz space forms.
\newblock {\em J. Differential Geom.}, 21(2):301--308, 1985.

\bibitem{GueritaudKassel}
Fran{\c c}ois Gu\'eritaud and Fanny Kassel.
\newblock Maximally stretched laminations on geometrically finite hyperbolic
  manifolds.
\newblock 2013.
\newblock arXiv:1307.0250.

\bibitem{Kassel08}
Fanny Kassel.
\newblock Proper actions on corank-one reductive homogeneous spaces.
\newblock {\em J. Lie Theory}, 18(4):961--978, 2008.

\bibitem{KasselThese}
Fanny Kassel.
\newblock {\em Quotients compacts d'espaces homog\`enes r\'eels ou
  $p$-adiques}.
\newblock PhD thesis, Universit\'e de Paris-Sud 11, 2009.

\bibitem{KimKim13}
Sungwoon Kim and Inkang Kim.
\newblock On deformation spaces of nonuniform hyperbolic lattices.
\newblock 2013.
\newblock arXiv:1310.1154.

\bibitem{Klingler96}
Bruno Klingler.
\newblock Compl\'etude des vari\'et\'es lorentziennes \`a courbure constante.
\newblock {\em Math. Ann.}, 306(2):353--370, 1996.

\bibitem{Kobayashi93}
Toshiyuki Kobayashi.
\newblock On discontinuous groups acting on homogeneous spaces with noncompact
  isotropy subgroups.
\newblock {\em J. Geom. Phys.}, 12(2):133--144, 1993.

\bibitem{Kulkarni81}
Ravi~S. Kulkarni.
\newblock Proper actions and pseudo—{R}iemannian space forms.
\newblock {\em Advances in Mathematics}, 40(1):10--51, 1981.

\bibitem{KulkarniRaymond85}
Ravi~S. Kulkarni and Frank Raymond.
\newblock {$3$}-dimensional {L}orentz space-forms and {S}eifert fiber spaces.
\newblock {\em J. Differential Geom.}, 21(2):231--268, 1985.

\bibitem{LabourieMSRI}
Fran{\c c}ois Labourie.
\newblock {C}hern--{S}imons invariant and {T}holozan volume formula.
\newblock Lecture in the {\it {L}orentzian geometric structures seminar}, MSRI,
  may 2015.

\bibitem{Margulis91}
Gregory~A. Margulis.
\newblock {\em Discrete subgroups of semisimple Lie groups}, volume~17.
\newblock Springer, 1991.

\bibitem{Mostow68}
George~D. Mostow.
\newblock Quasi-conformal mappings in $n$-space and the rigidity of hyperbolic
  space forms.
\newblock {\em Publ. Math. Inst. Hautes \'Etudes Sci.}, 34(1):53--104, 1968.

\bibitem{Salein00}
Fran{\c c}ois Salein.
\newblock Vari\'et\'es anti-de {S}itter de dimension 3 exotiques.
\newblock {\em Ann. Inst. Fourier}, 50(1):257--284, 2000.

\bibitem{Tholozan3}
Nicolas Tholozan.
\newblock Dominating surface group representations and deforming closed {A}nti
  de-{S}itter $3$-manifolds.
\newblock 2014.
\newblock arXiv:1307.3315.

\bibitem{TholozanThese}
Nicolas Tholozan.
\newblock {\em Uniformisation des variétés pseudo-riemanniennes localement
  homogènes}.
\newblock PhD thesis, Universit\'e de Nice Sophia-Antipolis, 2014.

\bibitem{Troyanov91}
Marc Troyanov.
\newblock Prescribing curvature on compact surfaces with conical singularities.
\newblock {\em Trans. Amer. Math. Soc.}, 324(2):793--821, 1991.

\end{thebibliography}

\end{document}